\documentclass{amsart}
\usepackage{amssymb}
\usepackage{amsmath}
\usepackage{setspace}
\usepackage[super]{nth}
\usepackage{xcolor}
\newtheorem{theorem}{Theorem}[section]
\newtheorem{proposition}[theorem]{Proposition}
\newtheorem{lemma}[theorem]{Lemma}

\theoremstyle{definition}

\newtheorem{conjecture}[theorem]{Conjecture}

\newcommand{\Fp}{\mathbb F_p}
\newcommand{\Dbc}{D^b_c}
\newcommand{\Ql}{\overline{\mathbb Q_\ell}}
\newcommand{\CC}{\mathbb C}
\newcommand{\LL}{\mathcal L}
\newcommand{\FF}{\mathcal F}
\newcommand{\HH}{\mathcal H}

\numberwithin{equation}{section}

\begin{document}
\title[On some conjectures on Generalized quadratic Gauss sums]
{On some conjectures on Generalized quadratic Gauss sums and related problems}
\author{Nilanjan Bag}
 \address{Department of Mathematics, Harish-Chandra Research Institute, HBNI, Chhatnag Road, Jhunsi, Prayagraj (Allahabad) - 211 019, India}
\email{nilanjanbag@hri.res.in}
\author{Antonio Rojas-Le\'{o}n}
 \address{Departament of Algebra, Universidad de Sevilla, c/Tarfia, s/n, 41012 Sevilla, Spain}
\email{arojas@us.es}
\author{Zhang Wenpeng}
 \address{School of Mathematics, Northwest University, Xi'an, 710127, Shaanxi, P. R. China}
\email{wpzhang@nwu.edu.cn}
\date{24th May, 2021}
\begin{abstract}
The main purpose of this article is to study higher power mean values of generalized quadratic Gauss sums using estimates for character sums, analytic method and algebraic geometric methods. In this article, we prove two conjectures which were proposed in \cite{BLZ}.
\end{abstract}
\maketitle
\section{Introduction and statements of the results}
Let $q\geq 2$ be an integer, and let $\chi$ be a Dirichlet character modulo $q$ and $n$ be any integer. For Dirichlet character $\chi\bmod q$, the generalized $k$-th Gauss sums $G(n,k,\chi;q)$ is defined as
\begin{equation}
 G(n,k,\chi;q)=\sum_{a=1}^q\chi(a)e\left(\frac{na^k}{q}\right),\notag
\end{equation}
where $e(y)=e^{2\pi iy}$. If $k=2$, then $G(n,2,\chi;q)$ is called the generalized quadratic Gauss sums and we denote it as $G(n,\chi;q)$. This sum generalizes the classical quadratic Gauss sum $G(n;q)$, which is defined as
\begin{equation}
 G(n;q)=\sum_{a=1}^{q}e\left(\frac{na^2}{q}\right).\notag
\end{equation}
This kind of character sum has been a very active object of study for a long time. The values of $G(n,\chi;q)$ behave irregularly whenever $\chi$ varies. For a positive integer $n$ with $\gcd(n,q)=1$, one can find a non-trivial upper bound of $|G(n,\chi;q)|$. For example, one can prove the upper bound estimate (see T. Cochrane, Z. Y. Zheng \cite{cochrane})
\begin{align*}
|G(n,\chi;q)|\leq 2^{\omega (q)}\sqrt{q},
\end{align*}
where $\omega(q)$ denotes the number of distinct prime divisors of $q$. 
In case of prime $p$, such bounds can be found in \cite{weil} and \cite{weil-2}. Let $p$ be an odd prime and $L(s,\chi)$ denote the Dirichlet $L$-function corresponding to the character $\chi \bmod p$. Let $\chi_0$ denote the principal character modulo $p$.
\par In \cite{zhang}, Zhang proposed two problems which are dedicated to finding asymptotic formulas for
 \begin{align*}
 \sum_{\chi\bmod p}|G(n,\chi;p)|^{2m} \text{ and }  \sum_{\chi \neq \chi_0}|G(n,\chi;p)|^{2m}|L(1,\chi)|
 \end{align*}
 for a general integer $m\geq 3$. Also in the same article he conjectured that,
\begin{conjecture}\label{C1}
For all positive integer $m$,
\begin{align*}
\sum_{\chi\neq\chi_0}|G(n,\chi;p)|^{2m}\cdot |L(1,\chi)|\sim C\sum_{\chi \bmod p}|G(n,\chi;p)|^{2m}, \qquad p\rightarrow+\infty,
\end{align*}
where
\begin{align}\label{constant-c}
C=\prod_p\left[1+\frac{\binom{2}{1}^2}{4^2.p^2}+\frac{\binom{4}{2}^2}{4^4.p^4}+\cdots+\frac{\binom{2m}{m}^2}{4^{2m}.p^{2m}}+\cdots\right]
\end{align}
is a constant and $\displaystyle \prod_p$ denotes the product over all primes.
\end{conjecture} Here $\displaystyle \sum_{\chi \bmod p}$ denotes the sum over all Dirichlet characters modulo $p$ and
 $\displaystyle \sum_{\chi \neq \chi_0}$ denotes the sum over all non-principal Dirichlet characters modulo $p$.
Zhang \cite{zhang} studied the hybrid power mean value involving the generalized Quadratic Gauss sums. He used estimates for character sums and analytic methods
to study second, fourth and sixth power mean values of generalized quadratic Gauss sums. To be specific, he proved that for any
integer $n$ with $\gcd(n,p)=1$
\begin{align*}
\sum_{\chi\bmod p}|G(n,\chi;p)|^4=\begin{cases}
(p-1)[3p^2-6p-1+4\left(\frac{n}{p}\right)\sqrt{p}], &\text{ if } p\equiv 1 \bmod 4;\\
(p-1)(3p^2-6p-1), &\text{ if } p\equiv 3 \bmod 4,
\end{cases}
\end{align*}
and
\begin{align*}
\sum_{\chi\bmod p}|G(n,\chi;p)|^6=(p-1)(10p^3-25p^2-4p-1), \text{ if } p\equiv 3 \bmod 4,
\end{align*}
where $\left(\frac{\bullet}{p}\right)$ is the Legendre symbol.
Later, He and Liao \cite{yuan} evaluated the $6$-th power mean value  when $p\equiv 1\bmod{4}$ and also in the same article they obtained the $8$-th power mean value of generalized quadratic Gauss sums for any odd prime $p$. To be specific, for any integer $n$ with $\gcd(n,p)=1$,
they \cite[Theorem 2 and 3]{yuan} proved that
\begin{align*}
\sum_{\chi\bmod p}|G(n,\chi;p)|^6=
\left\{
\begin{array}{ll}
(p-1)(10p^3-25p^2-16p-1)+(p\sqrt{p}(p-1)N\\
+18p^2\sqrt{p}-12p\sqrt{p}-6\sqrt{p})\left(\frac{n}{p}\right), & \hspace{-1.8cm} \hbox{if $p\equiv 1 \bmod 4$;} \\
(p-1)(10p^3-25p^2-4p-1), & \hspace{-1.8cm} \hbox{if $p\equiv 3 \bmod 4$,}
\end{array}
\right.
\end{align*}
where
\begin{align}\label{X1}
N=\sum_{a=2}^{p-2}\sum_{c=1}^{p-1}\left(\frac{a^2-c^2}{p}\right)\left(\frac{c^2-1}{p}\right)\left(\frac{a^2-1}{p}\right),
\end{align}
and
\begin{align*}
&\sum_{\chi\bmod p}|G(n,\chi;p)|^8\\
&=
\left\{
\begin{array}{ll}
(p-1)(34p^4-99p^3-65p^2-29p-1)\\+(56p^3\sqrt{p}+8p^2\sqrt{p}-56p\sqrt{p}-8\sqrt{p}+8p^2\sqrt{p}(p-1)N)\left(\frac{n}{p}\right)\\+p^2(p-1)T,
& \hspace{-1.8cm} \hbox{if $p\equiv 1 \bmod 4$;} \\
(p-1)(34p^4-99p^3+7p^2-5p-1)+p^2(p-1)T, & \hspace{-1.8cm} \hbox{if $p\equiv 3 \bmod 4$,}
\end{array}
\right.
\end{align*}
where $N$ is the same as \eqref{X1} and
\begin{align*}
T=\sum_{a=2}^{p-2}\sum_{b=1}^{p-1}\sum_{d=1}^{p-1}\left(\frac{a^2-b^2}{p}\right)\left(\frac{b^2-1}{p}\right)\left(\frac{a^2-d^2}{p}\right)\left(\frac{d^2-1}{p}\right).
\end{align*}
In article \cite{BB}, the first author and Barman derived asymptotic formulas for $T$ and $N$ which allowed them to get improved estimates for He and Liao's results. In particular, for odd prime $p$ and for any integer $n$ with $\gcd(n,p)=1$, they \cite[Theorem 1.7 and 1.8]{BB}
 proved that
\begin{align*}
\sum_{\chi \bmod p} \left|G(n,\chi;p)\right|^6
=\begin{cases}
 10p^4+O(p^{7/2}),&\text{if}~p\equiv 1\bmod 4;\\
 (p-1)(10p^3-25p^2-4p-1),&\text{if}~p\equiv 3\bmod 4,
 \end{cases}
\end{align*}
and
\begin{align*}
\sum_{\chi \bmod p} \left|G(n,\chi;p)\right|^8 =35p^5+O(p^{9/2}).
 \end{align*}
 The estimates for the $6$-th and $8$-th power mean values along with the results for $6$-th and $8$-th order moments of generalized quadratic Gauss sums weighted by $L$-functions proved Conjecture \ref{C1} upto $m\leq 4$.
%%%%%%%%%%%%%%%%%%%%%%%%%%%%%%%%%%%%%%%%%%%%%%%%%%%%%
Later we \cite{BLZ} derived an asymptotic formula for the $10$-th power mean value of the genralized quadratic Gauss sums and as an application we also derived an asymptotic formula for the $10$-th order moment of the generalized quadratic Gauss sums weighted by $L$-functions. We \cite[Theorem 1.4 and 1.5]{BLZ} proved, for any odd prime $p$ and any inetger $n$ with $\gcd (n,p)=1$
\begin{align*}
&\sum_{\chi\bmod p}|G(n,\chi;p)|^{10}=126\cdot p^6+O(p^{11/2})
\end{align*}
and
\begin{align*}
\sum_{\chi\neq\chi_0}\left|G(n,\chi;p)\right|^{10}\cdot |L(1,\chi)|=126\cdot C\cdot p^6+ O\left(p^{{11}/{2}}\cdot \ln^2 p\right),
\end{align*}
where $C$ is defined same as in (1.1). This proved Conjecture \ref{C1} upto $m\leq 5$.
\par In \cite{BLZ}, for any positive integer $m$, we also gave two conjectures for the general case.
\begin{conjecture}\cite[Conjecture 1.8]{BLZ}\label{C2} Let $p$ be a prime large enough, $m$ be any positive integer. Then we have the asymptotic formula
\begin{eqnarray*}
\frac{1}{p^{m+1}}\cdot \sum_{\chi \bmod p}\left|\sum_{a=1}^{p-1}\chi(a) e\left(\frac{na^2}{p}\right)\right|^{2m}=\binom{2m-1}{m} + o(1).
\end{eqnarray*}
\end{conjecture}
\begin{conjecture}\cite[Conjecture 1.9]{BLZ}\label{C3} Let $p$ be a prime large enough, $m$ be any positive integer. Then for any integer $n$ with $(n, p)=1$, we have the asymptotic formula
\begin{eqnarray*}
\frac{1}{p^{m+1}}\cdot\sum_{\chi\neq\chi_0}\left|\sum_{a=1}^{p-1}\chi(a) e\left(\frac{na^2}{p}\right)\right|^{2m}\cdot |L(1,\chi)|=\binom{2m-1}{m}\cdot C+ o\left(1\right),
\end{eqnarray*}
where $C$ is same as \eqref{constant-c}.
\end{conjecture}
 
 In this article, we first prove the following asymptotic formula:
\begin{theorem}\label{mt1}
Let $p$ be a odd prime large enough. Then for any positive integer $r$ we have 
\begin{align*}
\sum_{\chi\bmod p}\left(\sum_{a=1}^{p-1}\chi(a) \left(\frac{a^2-1}{p}\right)\right)^{r}=A_r\cdot p^{(r+2)/2} + O(p^{(r+1)/2}).
\end{align*}
Here, whenever $p\equiv 1\bmod 4$
\begin{align*}
A_r=\begin{cases}
\frac{1}{2}\cdot \binom{r}{r/2}, ~~&\text{if}~~ r ~\text{is even;}\\
0, ~~&\text{if}~~ r ~\text{is odd}
\end{cases}
\end{align*}
and whenever $p\equiv 3 \bmod 4$
\begin{align*}
A_r=\begin{cases}
(-1)^{r/2}\cdot \frac{1}{2}\cdot \binom{r}{r/2}, ~~&\text{if}~~ r ~\text{is even;}\\
0, ~~&\text{if}~~ r ~\text{is odd}.
\end{cases}
\end{align*}
\end{theorem}
With the help of Theorem \ref{mt1}, we prove conjecture \ref{C2}. In particular, we prove the following:
\begin{theorem}\label{mt2}
 Let $p$ be a prime large enough. Then for any positive integer $m$, we have the asymptotic formula
\begin{eqnarray*}
\sum_{\chi\bmod p}\left|\sum_{a=1}^{p-1}\chi(a) e\left(\frac{na^2}{p}\right)\right|^{2m}=\binom{2m-1}{m}\cdot p^{m+1} + O\left(p^{m+1/2}\right).
\end{eqnarray*}
\end{theorem}
%
%\begin{theorem}\label{mt3}
% Let $p$ be a prime large enough. Then for any positive integer $k$, we have the asymptotic formula
%\begin{eqnarray*}
%\cdot\sum_{\chi\neq\chi_0}\left|\sum_{a=1}^{p-1}\chi(a) e\left(\frac{na^2}{p}\right)\right|^{2m}\cdot |L(1,\chi)|=\binom{2m-1}{m}\cdot C \cdot p^{m+1} + O\left(p^{m+1/2}\right),
%\end{eqnarray*}
%where $C$ is same as \eqref{constant-c}.
%\end{theorem}
%\par Combining Theorem \ref{mt2} and Theorem \ref{mt3}, we prove Conjecture \ref{C1} for any positive integer $m$.
 In \cite{BLZ}, we also gave two corresponding conjectures.
 \begin{conjecture} \cite[Conjecture 1.10]{BLZ}\label{C4} Let $p$ be a prime large enough. Then for any positive integer $m$, we have the asymptotic formula
 \begin{align*}
 \frac{1}{p^{m+1}}\cdot \sum_{\substack{\chi \bmod p\\ \chi\neq\chi_0}}\left|\sum_{a=1}^{p-1}\chi(a+\overline{a})\right|^{2m}=\binom{2m-1}{m}+ o\left(1\right).
 \end{align*}
 \end{conjecture} 
  \begin{conjecture}\cite[Conjecture 1.11]{BLZ}\label{C5}
 Let $p$ be a prime large enough. Then for any positive integer $m$, we have the asymptotic formula
 \begin{align*}
 \frac{1}{p^{m+1}}\cdot \sum_{\substack{\chi \bmod p\\ \chi\neq\chi_0}}\left|\sum_{a=1}^{p-1}\chi(a+\overline{a})\right|^{2m}\cdot |L(1,\chi)|=\binom{2m-1}{m}\cdot C+ o\left(1\right),
 \end{align*}
 where $C$ is same as \eqref{constant-c}.
 \end{conjecture} 
 In this article, we prove Conjecture \ref{C4}. In particular, we prove the following:
\begin{theorem}\label{mt3}
 Let $p$ be a prime large enough. Then for any positive integer $m$, we have the asymptotic formula
 \begin{align*}
 \sum_{\substack{\chi \bmod p\\ \chi\neq\chi_0}}\left|\sum_{a=1}^{p-1}\chi(a+\overline{a})\right|^{2m}=\binom{2m-1}{m}\cdot p^{m+1}+O(p^{m+1/2}).
 \end{align*}
 \end{theorem} 

Conjecture \ref{C3} and Conjecture \ref{C5} seem to be more difficult and are subject of our forthcoming work.
 
\section{Proof of Theorem \ref{mt1}}
In this section we will use $\ell$-adic cohomology techniques to prove Theorem \ref{mt1}. Fix a prime $\ell\neq p$, and consider the category of $\ell$-adic constructible sheaves on the torus ${\mathbb G}_{m,\Fp}$ and its derived category $\Dbc({\mathbb G}_{m,\Fp},\Ql)$.

Denote by $\rho(t)$ the Legendre symbol $\left(\frac{t}{p}\right)$. First of all, note that we can write
\begin{align*}
S_r &:=\sum_{\chi\bmod p}\left(\sum_{a=1}^{p-1}\chi(a)\rho(a^2-1)\right)^r\\
&=\sum_{\chi\bmod p}\sum_{a_1,\ldots,a_r=1}^{p-1}\chi(a_1\cdots a_r)\rho(a_1^2-1)\cdots\rho(a_r^2-1)\\
&=(p-1)\sum_{\stackrel{a_1,\ldots,a_r=1}{\prod_{i=1}^r a_i=1}}^{p-1}\rho(a_1^2-1)\cdots\rho(a_r^2-1).
\end{align*}

Let $\LL_\rho$ be the Kummer sheaf on ${\mathbb G}_{m,\Fp}$ associated to $\rho$ \cite[1.4-1.7]{deligne} (extended by zero to the affine line ${\mathbb A}^1_{\Fp}$), and $\FF$ its pull-back via the map ${\mathbb A}^1_{\Fp}\to{\mathbb A}^1_{\Fp}$ given by $t\mapsto t^2-1$. For any $t\in\Fp$, a geometric Frobenius element at $t$ acts on the stalk of $\FF$ by multiplication by $\rho(t^2-1)$. Then our sum is $p-1$ times the trace of the action of a geometric Frobenius element at $t=1$ on the stalk of the $r$-fold (multiplicative) !-convolution object $\FF\ast\cdots\ast\FF\in\Dbc({\mathbb G}_{m,\Fp},\Ql)$, defined as ${\mathrm R}\mu_!(\FF\boxtimes\cdots\boxtimes\FF)$, where $\mu:{\mathbb G}_{m,\Fp}^r\to{\mathbb G}_{m,\Fp}$ is the multiplication map.

Let $P=\FF[1]\in\Dbc({\mathbb G}_{m,\Fp},\Ql)$ be the shifted object. Since $\FF$ does not have punctual sections, $P$ is a perverse sheaf \cite[2.3 (especially 2.3.6)]{katz-rls}. Moreover, $\FF$ is clearly not isomorphic to a Kummer sheaf, so $P$ has ``property $\mathcal P_!$'' in the terminology of \cite[2.6]{katz-rls} by \cite[Lemma 2.6.14]{katz-rls}. In particular, the $r$-fold convolution $P^{\ast r}:=P\ast\cdots\ast P=(\FF\ast\cdots\ast\FF)[r]$ is also a perverse sheaf and, as such, all its cohomology sheaves vanish except for the $-1$ and $0$-th ones, and the $0$-th one is punctual \cite[2.6.8]{katz-rls}. Therefore we get
$$
\frac{1}{p-1}S_r=\mathrm{Tr}(F_1|(\FF^{\ast r})_1)=(-1)^r\mathrm{Tr}(F_1|(\HH^r(\FF^{\ast r}))_1)+(-1)^{r-1}\mathrm{Tr}(F_1|(\HH^{r-1}(\FF^{\ast r}))_1).
$$

Since $\FF$ is pure of weight zero (all eigenvalues of the Frobenius actions on its stalks are roots of unity), $\HH^{r-1}(\FF^{\ast r})={\mathrm R}^{r-1}\mu_!(\FF\boxtimes\cdots\boxtimes\FF)$ is mixed of weights $\leq r-1$ by \cite[Theorem 3.3.1]{deligne-weil2}. Therefore, $|\mathrm{Tr}(F_1|(\HH^{r-1}(\FF^{\ast r}))_1)|\leq\dim(\HH^{r-1}(\FF^{\ast r})_1)\cdot p^{(r-1)/2}$.

\begin{lemma}
 We have the upper bound $\dim(\HH^{r-1}(\FF^{\ast r})_1)\leq r\cdot 2^{r-1}$.
\end{lemma}

\begin{proof}
 Since $\FF^{\ast r}[r]$ is perverse, $\HH^{r-1}(\FF^{\ast r})$ does not have punctual sections, so the dimension of its stalk at 1 is upper bounded by its generic rank.
 We will prove by induction that the generic rank of $\HH^{r-1}(\FF^{\ast r})$ is $\leq r\cdot 2^{r-1}$.
 
 For $r=1$ it is clear. Assuming it true for some $r-1$, we have
 \begin{align*}
 \mathrm{gen.rk.}\HH^{r-1}(\FF^{\ast r})
 &= \mathrm{gen.rk.}\HH^{-1}(\FF^{\ast r}[r])\\
 &= \mathrm{gen.rk.}\HH^{-1}(\FF^{\ast (r-1)}[r-1]\ast\FF[1])\\
 &\leq \chi_c(\FF^{\ast (r-1)}[r-1])\cdot\mathrm{gen.rk.}\HH^{-1}(\FF[1])\\&\qquad\qquad\qquad+\chi_c(\FF[1])\cdot\mathrm{gen.rk.}\HH^{-1}(\FF^{\ast (r-1)}[r-1])\\
 &=\chi_c(\FF[1])^{r-1}\cdot\mathrm{gen.rk.}\HH^{-1}(\FF[1])\\
 &\qquad\qquad\qquad+\chi_c(\FF[1])\cdot\mathrm{gen.rk.}\HH^{-1}(\FF^{\ast (r-1)}[r-1])\\
 &=2^{r-1}+2(r-1)\cdot 2^{r-2}\\
 &=r\cdot 2^{r-1},
 \end{align*}
% $$
% \mathrm{gen.rk.}\HH^{r-1}(\FF^{\ast r})=
% \mathrm{gen.rk.}\HH^{-1}(\FF^{\ast r}[r])=
% \mathrm{gen.rk.}\HH^{-1}(\FF^{\ast (r-1)}[r-1]\ast\FF[1])\leq
% $$
% $$
% \leq\chi_c(\FF^{\ast (r-1)}[r-1])\cdot\mathrm{gen.rk.}\HH^{-1}(\FF[1])+\chi_c(\FF[1])\cdot\mathrm{gen.rk.}\HH^{-1}(\FF^{\ast (r-1)}[r-1])=
% $$
% $$
% =\chi_c(\FF[1])^{r-1}\cdot\mathrm{gen.rk.}\HH^{-1}(\FF[1])+\chi_c(\FF[1])\cdot\mathrm{gen.rk.}\HH^{-1}(\FF^{\ast (r-1)}[r-1])=
% $$
% $$
% =2^{r-1}+2(r-1)\cdot 2^{r-2}=r\cdot 2^{r-1}
%$$ 
where $\chi_c$ is the Euler characteristic (with compact supports), by \cite[Theorem 28.2]{katz-ce}.
\end{proof}

As a consequence, we get
$$
\left|\frac{1}{p-1}S_r-(-1)^r\mathrm{Tr}(F_1|(\HH^r(\FF^{\ast r}))_1)\right|=|\mathrm{Tr}(F_1|(\HH^{r-1}(\FF^{\ast r}))_1)|\leq r\cdot 2^{r-1}\cdot p^{(r-1)/2},
$$
so Theorem \ref{mt1} follows from the following
\begin{proposition}
 We have
 $$
 \mathrm{Tr}(F_1|(\HH^r(\FF^{\ast r}))_1)=\left\{\begin{array}{ll} 
 0, & \mbox{if $r$ is odd;} \\
 \left(\frac{-1}{p}\right)^{r/2}\frac{1}{2}{{r}\choose{r/2}} p^{r/2}, & \mbox{if $r$ is even.}
    \end{array}\right.
$$
\end{proposition}

\begin{proof}
 As explained in \cite{katz-ce}, the category of all subquotients (in the category of perverse sheaves in ${\mathbb G}_{m,\Fp}$) of all the convolution powers of a semisimple perverse sheaf $\mathcal P$ pure of weight $0$ and its dual (modulo negligible objects, that is, objects which are succesive extensions of Kummer sheaves) is a finitely generated Tannakian neutral category with respect to the convolution opearation, where the dimension of an object is its Euler charactetistic, the identity object is the punctual sheaf $\delta_1[0]$ and the dual of an object $K$ is $\iota^\ast DK$, where $\iota$ is the inversion map $t\mapsto t^{-1}$ and $DK$ is the usual Verdier dual of $K$. It is therefore equivalent to the category of representations of a certain reductive algebraic group $G_{arith}\subseteq\mathrm{GL}_n(\Ql)\cong\mathrm{GL}_n(\mathbb C)$ with the tensor product operation, where $n=\chi_c({\mathbb G}_{m,\overline\Fp},\mathcal P)$. Similarly, for the pull-back of $\mathcal P$ to ${\mathbb G}_{m,\overline\Fp}$ we obtain a reductive group $G_{geom}\subseteq\mathrm{GL}_n(\mathbb C)$, which is a normal subgroup of $G_{arith}$ \cite[Theorem 6.1]{katz-ce}.
 
 We will apply this theory to a suitable twist of the object $\FF[1]$. First we fix a square root of $p$ so that we can do a half-Tate twist to get $\FF[1](1/2)$, which is pure of weight 0. Since $\rho$ is self-conjugate, $\FF[1](1/2)$ is its own Verdier dual, so its dual in the Tannakian sense is $\iota^\ast\FF[1](1/2)$. The trace of the action of a Frobenius element at a point $t\in {\mathbb F}_{p^n}^\ast$ (for ${\mathbb F}_{p^n}$ a finite extension of $\Fp$) on $\iota^\ast\FF$ is
 $$
 \rho\left(\mathrm{Nm}_{{\mathbb F}_{p^n}/\Fp}\left(\frac{1}{t^2}-1\right)\right)=\rho(-1)^n\rho(\mathrm{Nm}_{{\mathbb F}_{p^n}/\Fp}({t^2}-1)).
 $$
 By the Chebotarev density theorem, these traces determine a semisimple sheaf up to isomorphism, so we conclude that $\iota^\ast\FF$ is isomorphic to $(\rho(-1))^{deg}\otimes\FF$. If we take a square root $\epsilon$ of $\rho(-1)$ (so, for instance, $\epsilon=1$ for $p\equiv 1\bmod 4$ and $\epsilon=i$ for $p\equiv 3\bmod 4$), we deduce that $K:=\epsilon^{deg}\otimes\FF[1](1/2)$ is self-dual in the Tannakian sense.
 
 Therefore, either its $G_{arith}\subseteq \mathrm{O}_2(\mathbb C)$ or $G_{arith}\subseteq \mathrm{Sp}_2(\mathbb C)$. Now, by \cite[Lemma 23.1]{katz-ce}, the determinant of $G_{geom}$ corresponds to $\delta_{-1}[0]$ (note that, even though the paragraph before the lemma assumes $n\geq 3$, this fact is not used in the proof), which is a character of order $2$ (since $\delta_{-1}[0]\ast\delta_{-1}[0]=\delta_1[0]$). This rules out the case $G_{geom}\subseteq \mathrm{Sp}_2(\mathbb C)=\mathrm{SL}_2(\mathbb C)$. Moreover, since $\FF[1]$ is not punctual, $G_{geom}$ is not finite by \cite[Theorem 6.4]{katz-ce}. Therefore $G_{geom}$ is a positive-dimensional closed subgroup of the $1$-dimensional $\mathrm{O}_2(\CC)$, so it is either $\mathrm{O}_2(\CC)$ or $\mathrm{SO}_2(\CC)$. Using again that the determinant is geometrically non-trivial, we conclude that $G_{geom}=\mathrm{O}_2(\CC)$, which forces $G_{arith}$ to be $\mathrm{O}_2(\CC)$ too.
 
 Now $\HH^r(\FF^{\ast r})_1=\HH^0(\FF^{\ast r}[r])_1$ is the punctual component at $t=1$ of $\FF^{\ast r}[r]$. Looking at it from the other side of the Tannakian correspondence (geometrically for now), the punctual object $\delta_1$ corresponds to the trivial representation of $G_{geom}$, so the dimension of $\HH^r(\FF^{\ast r})_1$ is the dimension of the $G_{geom}$-invariant part of the representation $V^{\otimes r}$ of $G_{geom}=\mathrm{O}_2(\CC)$, where $V\cong\CC^2$ is the space on which the standard representation $\sigma$ of $G_{geom}$ (corresponding to $\FF[1]$) acts.
 
 The category of representations of $\mathrm{O}_2(\CC)$ as an algebraic group is equivalent to the category of representations of its maximal compact subgroup $\mathrm{O}_2(\mathbb R)$ as a compact Lie group. The latter can be parameterized by the disjoint union of the intervals $[0,2\pi)\sqcup [0,2\pi)$ via the maps
 $$
 \phi_+:[0,2\pi)\to\mathrm{SO}_2(\mathbb R);\;t\mapsto\begin{pmatrix} \cos(t) & -\sin(t) \\ \sin(t) & \cos(t)\end{pmatrix}
 $$
 $$
 \phi_-:[0,2\pi)\to\mathrm{O}^-_2(\mathbb R);\;t\mapsto\begin{pmatrix}\cos(t) & \sin(t) \\ \sin(t) & -\cos(t)\end{pmatrix}
$$
under which the Haar measure $\mu_{Haar}$ on $\mathrm{O}_2(\mathbb R)$ corresponds to the normalized Lebesgue measure $\frac{1}{4\pi}\mu_{Leb}$ on $[0,2\pi)\sqcup [0,2\pi)$. The dimension of the trivial part of $V^{\otimes r}$ is then
\begin{align*}
\int_{\mathrm{O}_2(\mathbb R)} \mathrm{Tr}(\sigma^{\otimes r}) d\mu_{Haar}&=
\int_{\mathrm{O}_2(\mathbb R)} \mathrm{Tr}(\sigma)^r d\mu_{Haar}\\
&=\frac{1}{4\pi}\int_0^{2\pi}(2\cos(t))^rdt+\frac{1}{4\pi}\int_0^{2\pi} 0^rdt\\
&=\frac{1}{4\pi}\int_0^{2\pi}(2\cos(t))^rdt,
\end{align*}
%$$
%\int_{\mathrm{O}_2(\mathbb R)} \mathrm{Tr}(\sigma^{\otimes r}) d\mu_{Haar}=
%\int_{\mathrm{O}_2(\mathbb R)} \mathrm{Tr}(\sigma)^r d\mu_{Haar}=
%$$
%$$
%=\frac{1}{4\pi}\int_0^{2\pi}(2\cos(t))^rdt+\frac{1}{4\pi}\int_0^{2\pi} 0^rdt=\frac{1}{4\pi}\int_0^{2\pi}(2\cos(t))^rdt
%$$
which is clearly $0$ for $r$ odd since $\cos(t+\pi)=-\cos(t)$. For $r$ even, integration by parts gives
$$
\int_0^{2\pi}(2\cos(t))^rdt=4\cdot\frac{r-1}{r}\int_0^{2\pi}(2\cos(t))^{r-2}dt
$$
and, by induction, we easily conclude that
$$
\int_{\mathrm{O}_2(\mathbb R)} \mathrm{Tr}(\sigma^{\otimes r}) d\mu_{Haar}=\frac{1}{2}{{r}\choose{r/2}}.
$$

Finally, since $G_{geom}=G_{arith}$, $G_{arith}$ acts trivially on the $G_{geom}$-invariant part. That is, it acts trivially on the punctual part $\HH^0(K^{\ast r})_1$ of $K^{\ast r}$ at $t=1$. Therefore, since $K^{\ast r}=\epsilon^{r\cdot deg}\otimes\FF^{\ast r}[r](r/2)$, it acts via multiplication by $\epsilon^r p^{r/2}=\rho(-1)^{r/2} p^{r/2}$ on the punctual part $\HH^0(\FF^{\ast r}[r]))_1$ and, in particular, the trace of the Frobenius action on $\HH^0(\FF^{\ast r}[r]))_1$ is 
$$
\rho(-1)^{r/2} p^{r/2}\cdot\frac{1}{2}{{r}\choose{r/2}}=\left(\frac{-1}{p}\right)^{r/2} p^{r/2}\cdot\frac{1}{2}{{r}\choose{r/2}}.$$
 \end{proof}

\section{Proofs of Theorem \ref{mt2} and Theorem \ref{mt3}}
For convenience, for the rest part of the article we denote $A=1+\chi(-1)$ and $B=\left(\frac{n}{p}\right)G(1;p)$,
where $G(1;p)$ is the Gauss sum $\displaystyle G(1;p)=\sum_{b=0}^{p-1}e\left(\frac{b^2}{p}\right)$ and $\left(\frac{\bullet}{p}\right)$ is the Legendre symbol. First we state a few lemmas which we will use in the proofs of Theorem \ref{mt2} and Theorem \ref{mt3}.
\begin{lemma}\cite[Lemma 2]{yuan}\label{X2}
 Let $p$ be an odd prime and $n$ be any integer with $\gcd(n,p)=1$. Then for any non-principal character  $\chi$ modulo $p$ the following identity holds
 \begin{equation*}
  |G(n,\chi;p)|^2=Ap+B\sum_{a=2}^{p-2}\chi(a)\left(\frac{a^2-1}{p}\right).
 \end{equation*}
 If $\chi_0$ is the principal character modulo $p$, then
 \begin{align*}
 |G(n,\chi_0;p)|^2=\begin{cases}
 p+1-2\sqrt{p}\left(\frac{n}{p}\right), &p\equiv 1\bmod 4;\\
 p+1, &p\equiv 3\bmod 4.
 \end{cases}
 \end{align*}
\end{lemma}

\begin{lemma}\cite[Section 9.10]{gauss}\label{X11}
 For any integer $q\geq1$, we have
 \begin{equation*}
  G(1;q)=\frac{1}{2}\sqrt{q}(1+i)(1+e^{\frac{-\pi iq}{2}})=\begin{cases}\sqrt{q} &\ \text{if}\quad  q\equiv1\bmod 4;\\
0 &\ \text{if}\quad  q\equiv2\bmod 4;\\
i\sqrt{q} &\ \text{if}\quad  q\equiv3\bmod 4;\\
(1+i)\sqrt{q}  &\ \text{if}\quad

 q\equiv0\bmod 4.
\end{cases}
  \end{equation*}
\end{lemma}
\begin{lemma}\cite[Lemma 2.8]{BLZ}\label{X12}
Let $p$ be an odd prime and $\chi$ be any non-principal character modulo $p$. Then for any integer $t$ with $\gcd(t,p)=1$, we have the identity
\begin{align*}
\left|\sum_{a=1}^{p-1}\chi(ta+\overline{a})\right|=\left|\sum_{a=1}^{p-1}\chi(a)\left(\frac{a^2-t}{p}\right)\right|.
\end{align*}
\end{lemma}
Now we give proofs of Theorem \ref{mt2} and Theorem \ref{mt3}.

\begin{proof}[Proof of Theorem \ref{mt2}]
\par For any $n$ with $\gcd(n,p)=1$, we have
 \begin{equation}\label{X20}
   \sum_{\chi\bmod p}|G(n,\chi;p)|^{2m}= \sum_{\substack{\chi\neq\chi_0}}|G(n,\chi;p)|^{2m}+|G(n,{\chi}_{0};p)|^{2m},
 \end{equation}
 where using Lemma \ref{X2}, we obtain
 \begin{align}\label{X21}
 |G(n,{\chi}_{0};p)|^{2m}=\begin{cases}
 \left( p+1-2\sqrt{p}\left(\frac{n}{p}\right)\right)^m, &p\equiv 1\bmod 4;\\
 (p+1)^m, &p\equiv 3\bmod 4.
 \end{cases}
 \end{align}
 Note that for any odd character $\chi\bmod p$ we have
 \begin{align*}
 G(n,\chi;p)=\sum_{a=1}^{p-1}\chi(a) e\left(\frac{na^2}{p}\right)=0.
 \end{align*}
Hence it follows from Lemma \ref{X2} that
\begin{align}
   \sum_{\chi\neq\chi_0}|G(n,\chi;p)|^{2m}&=
    \sum_{\substack{\chi\neq\chi_0\\ \chi(-1)=1}}|G(n,\chi;p)|^{2m}\notag\\
   &=\sum_{\substack{\chi\neq\chi_0\\ \chi(-1)=1}}\left(Ap+B\sum_{a=2}^{p-2}\chi(a)\left(\frac{a^2-1}{p}\right)\right)^m\notag\\
   &=\sum_{\substack{\chi\neq\chi_0\\ \chi(-1)=1}}\sum_{k=0}^{m}\binom{m}{k}(Ap)^{m-k}\left(B\sum_{a=2}^{p-2}\chi(a)\left(\frac{a^2-1}{p}\right)\right)^k\notag\\
   &=\sum_{k=0}^{m}p^{m-k}B^k\binom{m}{k}\notag\\&\qquad\qquad\sum_{\substack{\chi\neq\chi_0\\ \chi(-1)=1}}(1+\chi(-1))^{m-k}\left(\sum_{a=2}^{p-2}\chi(a)\left(\frac{a^2-1}{p}\right)\right)^k\notag\\
   &=\sum_{k=0}^{m}(2p)^{m-k}B^k\binom{m}{k}\sum_{\substack{\chi\neq\chi_0\\ \chi(-1)=1}}\left(\sum_{a=2}^{p-2}\chi(a)\left(\frac{a^2-1}{p}\right)\right)^k\notag\\
   &=\sum_{k=0}^{m}(2p)^{m-k}B^k\binom{m}{k}\sum_{\chi\neq\chi_0}\left(\sum_{a=2}^{p-2}\chi(a)\left(\frac{a^2-1}{p}\right)\right)^k\notag\\
   &=\sum_{k=0}^{m}(2p)^{m-k}B^k\binom{m}{k}\sum_{\chi\bmod p}\left(\sum_{a=2}^{p-2}\chi(a)\left(\frac{a^2-1}{p}\right)\right)^k\notag\\&\qquad\qquad\qquad\qquad\qquad\qquad\qquad\qquad+O(p^{m+1/2})\notag.
\end{align}
Using Theorem \ref{mt1} and Lemma \ref{X11}, we get
\begin{align}\label{th-2}
\sum_{\chi\neq\chi_0}|G(n,\chi;p)|^{2m}&=\sum_{k=0}^{m}(2p)^{m-k}B^k\binom{m}{k}\left(A_k.p^{(k+2)/2} + O(p^{(k+1)/2})\right)\notag\\
&\qquad\qquad\qquad\qquad\qquad\qquad\qquad\qquad+O(p^{m+1/2})\notag\\
&=\sum_{\substack{k=0\\ k~ \text{even}}}^{m}2^{m-k-1}\binom{m}{k}\binom{k}{k/2} p^{m+1}+ O(p^{m+1/2})\notag\\
&=\frac{1}{2}\cdot\sum_{\substack{k=0\\ k~ \text{even}}}^{m}2^{m-k}\binom{m}{k}\binom{k}{k/2} \cdot p^{m+1}+ O(p^{m+1/2}).
\end{align}
Now consider the polynomial $(1+x)^{2m}$ which can be rewritten as $(1+2x+x^2)^m$, where $(1+2x+x^2)^m$ has a binomial expansion
\begin{align}
(1+2x+x^2)^m&=\sum_{k=0}^{m}\binom{m}{k}(2x)^{m-k}(1+x^2)^{k}\notag\\
&=\sum_{k=0}^{m}\binom{m}{k}(2x)^{m-k}\sum_{r=0}^{k}\binom{k}{r}x^{2r}\notag\\
&=\sum_{k=0}^{m}\binom{m}{k}2^{m-k}\sum_{r=0}^{k}\binom{k}{r}x^{m-k+2r}\notag.
\end{align}

Hence comparing the coefficients of $x^m$, we get 
\begin{align*}
\sum_{\substack{k=0\\ k~ \text{even}}}^{m}2^{m-k}\binom{m}{k}\binom{k}{k/2}=\binom{2m}{m}.
\end{align*}
Thus from \eqref{th-2} we get
\begin{align*}
\sum_{\chi\neq\chi_0}|G(n,\chi;p)|^{2m}&=\frac{1}{2}\cdot\binom{2m}{m}\cdot p^{m+1} + O(p^{m+1/2})\\
&=\binom{2m-1}{m} \cdot p^{m+1}+ O(p^{m+1/2}).
\end{align*}
Hence combining \eqref{X20} and \eqref{X21}, we prove Theorem \ref{mt2}.
\end{proof}
\begin{proof}[Proof of Theorem \ref{mt3}]
From lemma \ref{X2}, Lemma \ref{X11} and Lemma \ref{X12} , with $t=1$, and then using Theorem \ref{mt2} we have
\begin{align}
 \sum_{\substack{\chi \bmod p\\ \chi\neq\chi_0}}\left|\sum_{a=1}^{p-1}\chi(a+\overline{a})\right|^{2m}&= \sum_{\substack{\chi \bmod p\\ \chi\neq\chi_0}}\left|\sum_{a=1}^{p-1}\chi(a)\left(\frac{a^2-1}{p}\right)\right|^{2m}\notag\\
 &=\frac{1}{|G(1;p)|^{2m}}\cdot \sum_{\substack{\chi(-1)=1\\ \chi\neq\chi_0}}\left(\left|\sum_{a=1}^{p-1}\chi(a) e\left(\frac{a^2}{p}\right)\right|^2-2p\right)^{2m}\notag\\
 &=\frac{1}{p^m}\cdot \sum_{k=0}^{2m}\binom{2m}{k}(-2p)^{2m-k}\sum_{\substack{\chi(-1)=1\\ \chi\neq\chi_0}}\left|\sum_{a=1}^{p-1}\chi(a) e\left(\frac{a^2}{p}\right)\right|^{2k}\notag\\
 &=p^{m+1}\cdot \sum_{k=0}^{2m}(-2)^{2m-k}\binom{2m}{k}\binom{2k-1}{k}+O(p^{m+1/2})\notag\\
 &=\frac{1}{2}\cdot p^{m+1}\cdot \sum_{k=0}^{2m}(-1)^{k}2^{2m-k}\binom{2m}{k}\binom{2k}{k}+O(p^{m+1/2})\notag.
\end{align}
Now notice that the polynomial $(1+x^2)^{2m}$ can be rewritten as $\left(2x+(1-x)^2\right)^{2m}$, where 
\begin{align*}
\left(2x+(1-x)^2\right)^{2m}&=\sum_{k=0}^{2m}\binom{2m}{k}(2x)^{2m-k}(1-x)^{2k}\\
&=\sum_{k=0}^{2m}\binom{2m}{k}(2x)^{2m-k}\sum_{r=0}^{2k}\binom{2k}{r}(-x)^{r}\\
&=\sum_{k=0}^{2m}\binom{2m}{k}2^{2m-k}\sum_{r=0}^{2k}\binom{2k}{r}(-1)^{r}x^{2m-k+r}.
\end{align*}
Hence comparing the coefficient of $x^{2m}$ we get
\begin{align*}
\binom{2m}{m}=\sum_{k=0}^{2m}(-1)^{k}2^{2m-k}\binom{2m}{k}\binom{2k}{k},
\end{align*}
which gives
\begin{align*}
\sum_{\substack{\chi \bmod p\\ \chi\neq\chi_0}}\left|\sum_{a=1}^{p-1}\chi(a+\overline{a})\right|^{2m}&=\frac{1}{2}\cdot  \binom{2m}{m}\cdot p^{m+1}+O(p^{m+1/2})\\
&= \binom{2m-1}{m}\cdot p^{m+1}+O(p^{m+1/2}).
\end{align*}
This completes the proof of Theorem \ref{mt3}.
\end{proof}

%%%%%%%%%%%%%%%%%%%%%%%%%%%%%%%%%%%%%%%%%%%%%%%%%%%%%%%%%%%%%%%%%%%%%%%%%%%%%%%%%%%%%%%%%%%%%%%
\section{Acknowledgement}
During the prepartion of this article N.B. was supported by the post-doctoral fellowship in Harish-Chandra Research Institute, Prayagraj, and A.R-L. was partially supported by grants MTM2016-75027-P (Ministerio de Economı\'{\i}a y
Competitividad and FEDER) and US-1262169 (Consejería de Econom\'{\i}a, Conocimiento, Empresas y Universidad de la Junta de Andalucía and FEDER).

We would like to thank N. Katz for his clarifications about some of the results in \cite{katz-ce}.
%%%%%%%%%%%%%%%%%%%%%%%%%%%%%%%%%%%%%%%%%%%%%%%%%%%%%%%%%%%%%%%%%%%%%%%%%%%%%%%%%%%%%%%%%%%%%%%%%%%%%%%%%%%

\end{document}